\documentclass[11pt]{article}
\usepackage{amsmath}
\usepackage{amssymb}
\usepackage{amsthm}
\usepackage[usenames]{color}
\usepackage{amscd}
\usepackage{dsfont}
\usepackage{indentfirst}

\usepackage[colorlinks=true,linkcolor=blue,filecolor=red,
citecolor=webgreen]{hyperref}
\definecolor{webgreen}{rgb}{0,.5,0}

\numberwithin{equation}{section}

\def\N{{\mathds{N}}}
\def\Z{{\mathds{Z}}}

\def\1{{\bf 1}}

\newtheorem{theorem}{Theorem}[section]

\newtheorem{lemma}[theorem]{Lemma}

\begin{document}

\title{{\bf Short proof and generalization of a Menon-type identity by Li, Hu and Kim}}
\author{L\'aszl\'o T\'oth \\ \\ Department of Mathematics, University of P\'ecs \\
Ifj\'us\'ag \'utja 6, 7624 P\'ecs, Hungary \\ E-mail: {\tt ltoth@gamma.ttk.pte.hu}}
\date{}
\maketitle

\centerline{Taiwanese Journal of Mathematics {\bf 23} (2019), 557--561}

\begin{abstract} We present a simple proof and a generalization of a Menon-type identity by Li, Hu and Kim, involving Dirichlet characters
and additive characters.
\end{abstract}



\section{Motivation and main result}

Menon's classical identity states that for every $n\in \N$,
\begin{equation} \label{Menon_id}
\sum_{\substack{a=1\\ (a,n)=1}}^n (a-1,n) = \varphi(n)\tau(n),
\end{equation}
where $(a-1,n)$ stands for the greatest common divisor of $a-1$ and $n$, $\varphi(n)$ is Euler's totient function and $\tau(n)=\sum_{d\mid n} 1$
is the divisor function. Identity \eqref{Menon_id} was generalized by several authors in various directions. Zhao and Cao \cite{ZhaCao}
proved that
\begin{equation} \label{Menon_id_char}
\sum_{a=1}^n (a-1,n) \chi(a)= \varphi(n)\tau(n/d),
\end{equation}
where $\chi$ is a Dirichlet character (mod $n$) with conductor $d$ ($n\in \N$, $d\mid n$). If $\chi$ is the principal character (mod $n$), that is
$d=1$, then \eqref{Menon_id_char} reduces to Menon's identity \eqref{Menon_id}. Generalizations of \eqref{Menon_id_char} involving even functions
(mod $n$) were deduced by the author \cite{Tot2018}, using a different approach.

Li, Hu and Kim \cite{LiHuKimTaiw} proved the following generalization of identity \eqref{Menon_id_char}:

\begin{theorem}[{\rm \cite[Th.\ 1.1]{LiHuKimTaiw}}] \label{Th_LiHuKim} Let $n\in \N$ and let $\chi$ be a Dirichlet character \textup{(mod $n$)} with conductor $d$ \textup{($d\mid n$)}.  Let $b\mapsto\lambda_{\ell}(b):= \exp(2\pi i w_{\ell}b/n)$ be
additive characters of the group $\Z_n$, with $w_{\ell}\in \Z$ \textup{($1\le \ell \le k$)}. Then
\begin{equation} \label{Menon_general_1_k}
\sum_{a,b_1,\ldots,b_k=1}^n (a-1,b_1,\ldots,b_k,n)\chi(a)\lambda_1(b_1)\cdots \lambda_k(b_k)  = \varphi(n) \sigma_k((n/d,w_1,\ldots,w_k)),
\end{equation}
where $\sigma_k(n)=\sum_{d\mid n} d^k$.
\end{theorem}

Note that in \eqref{Menon_id_char} and \eqref{Menon_general_1_k} the sums are, in fact, over $1\le a\le n$ with $(a,n)=1$, since
$\chi(a)=0$ for $(a,n)>1$.  In the case $w_1=\cdots =w_k=0$, identity \eqref{Menon_general_1_k} was deduced by the same authors in paper
\cite{LiHuKimArXiv}. For the proof, Li, Hu and Kim computed first the given sum in the case $n=p^t$, a prime power, and then they showed
that the sum is multiplicative in $n$.

It is the goal of this paper to present a simple proof of Theorem \ref{Th_LiHuKim}. Our approach is similar to that given in \cite{Tot2018}, and
leads to a direct evaluation of the corresponding sum for every $n\in \N$. We obtain, in fact, the following generalization
of the  above result. Let $\mu$ denote the M\"obius function and let $*$ be the convolution
of arithmetic functions.

\begin{theorem} \label{Th_gen_TL} Let $F$ be an arbitrary arithmetic function, let $s_j\in \Z$, $\chi_j$ be Dirichlet characters \textup{(mod $n$)}
with conductors $d_j$ \textup{($1\le j\le m$)} and $\lambda_{\ell}$ be additive characters as defined above, with $w_{\ell} \in \Z$
\textup{($1\le \ell \le k$)}. Then
\begin{gather} \nonumber
\sum_{a_1,\ldots,a_m, b_1,\ldots,b_k=1}^n F((a_1-s_1,\ldots, a_m-s_m, b_1,\ldots,b_k,n)) \chi_1(a_1)\cdots \chi_m(a_m)
\lambda_1(b_1)\cdots \lambda_k(b_k) \\
= \varphi(n)^m \chi_1^*(s_1)\cdots \chi_m^*(s_m) \sum_{\substack{e\mid (n/d_1,\ldots,n/d_m,w_1,\ldots,w_k)\\ (n/e,s_1\cdots s_m)=1}}
\frac{e^k(\mu*F)(n/e)}{\varphi(n/e)^m}, \label{Menon_general_m_k}
\end{gather}
where $\chi_j^*$ are the primitive characters \textup{(mod $d_j$)} that induce $\chi_j$ \textup{($1\le j\le m$)}.
\end{theorem}

We remark that the sum in the left hand side of identity \eqref{Menon_general_m_k} vanishes provided that there is an $s_j$ such
that $(s_j,d_j)>1$. If $F(n)=n$ ($n\in \N$), $m=1$ and $s_1=1$, then identity \eqref{Menon_general_m_k} reduces to \eqref{Menon_general_1_k}.
We also remark that the special case $F(n)=n$ ($n\in \N$), $m\ge 1$, $s_1=\cdots =s_m=1$, $k\ge 1$, $w_1=\cdots =w_k=0$ was considered in the
quite recent preprint \cite{ChenHuLiArXiv}. Several other special cases of formula \eqref{Menon_general_m_k} can be discussed.

See the papers \cite{LiHuKimArXiv,LiHuKimTaiw,Tot2011,Tot2018,ZhaCao} and the references therein for other generalizations and analogues
of Menon's identity.
\\
\section{Proof}

We need the following lemmas.

\begin{lemma} \label{Lemma_cong} Let $n,d,e\in \N$, $d\mid n$, $e\mid n$ and let $r,s\in \Z$. Then
\begin{equation*}
\sum_{\substack{a=1\\ (a,n)=1\\ a\equiv r \, \text{\rm (mod $d$)} \\ a\equiv s \, \text{\rm (mod $e$)} }}^n 1 =
\begin{cases} \displaystyle \frac{\varphi(n)}{\varphi(de)}(d,e),
& \text{ if $(r,d)=(s,e)=1$ and $(d,e) \mid r-s$}, \\ 0, & \text{ otherwise}.
\end{cases}
\end{equation*}
\end{lemma}

In the special case $e=1$ this is known in the literature, usually proved  by the
inclusion-exclusion principle. See, e.g., \cite[Th.\ 5.32]{Apo1976}. Here we use a different approach,
in the spirit of our paper.

\begin{proof}[Proof of Lemma {\rm \ref{Lemma_cong}}]
For each term of the sum, since $(a,n)=1$, we have $(r,d)=(a,d)=1$ and $(s,e)=(a,e)=1$. Also, the given congruences
imply $(d,e) \mid r-s$. We assume that these conditions are satisfied (otherwise the sum is empty and equals zero).

Using the property of the M\"{o}bius function, the given sum, say $S$, can be written as
\begin{equation} \label{last_sum}
S= \sum_{\substack{a=1\\ a\equiv r \, \text{\rm (mod $d$)} \\ a\equiv s \, \text{\rm (mod $e$)} }}^n  \sum_{\delta \mid (a,n)}
\mu(\delta) = \sum_{\delta \mid n} \mu(\delta) \sum_{\substack{j=1\\ \delta j\equiv r \, \text{\rm (mod $d$)} \\
\delta j\equiv s \, \text{\rm (mod $e$)} }}^{n/\delta} 1.
\end{equation}

Let $\delta\mid n$ be fixed. The linear congruence $\delta j\equiv r$ (mod $d$) has solutions in $j$ if and only if $(\delta,d)\mid r$, equivalent to
$(\delta,d)=1$, since $(r,d)=1$. Similarly, the congruence $\delta j\equiv s$ (mod $e$) has solutions in $j$ if and only if $(\delta,e)\mid s$,
equivalent to $(\delta,e)=1$, since $(s,e)=1$. These two congruences have common solutions in $j$ due to the condition $(d,e)\mid r-s$. Furthermore, if
$j_1$ and $j_2$ are solutions of these simultaneous congruences, then $\delta j_1\equiv \delta j_2$ (mod $d$) and $\delta j_1\equiv \delta j_2$
(mod $e$). Since $(\delta,d)=1$, this gives $j_1\equiv j_2$  (mod $[d,e]$). We deduce that there are
\begin{equation*}
N= \frac{n}{\delta[d,e]}
\end{equation*}
solutions (mod $n/\delta$) and the last sum in \eqref{last_sum} is $N$. This gives
\begin{equation*}
S= \frac{n}{[d,e]} \sum_{\substack{\delta \mid n\\ (\delta,de)=1}} \frac{\mu(\delta)}{\delta}= \frac{n}{[d,e]} \cdot
\frac{\varphi(n)/n}{\varphi(de)/(de)}= \frac{\varphi(n)}{\varphi(de)}(d,e).
\end{equation*}
\end{proof}

The next lemma is a known result. See, e.g., \cite{Tot2018} for its (short) proof.

\begin{lemma} \label{Lemma_char_sum}
Let $n\in \N$ and $\chi$ be a primitive character \textup{(mod $n$)}. Then for any $e\mid n$, $e<n$ and any $s\in \Z$,
\begin{equation*}
\sum_{\substack{a=1\\ a\equiv s \, \text{\rm (mod $e$)} }}^n \chi(a)=0.
\end{equation*}
\end{lemma}

Now we prove

\begin{lemma} \label{Lemma_charact} Let $\chi$ be a Dirichlet character \textup{(mod $n$)} with conductor $d$ \textup{($n\in \N$, $d\mid n$)}
and let $e\mid n$, $s\in \Z$. Then
\begin{equation*}
\sum_{\substack{a=1\\ a\equiv s \, \text{\rm (mod $e$)} }}^n \chi(a) =
\begin{cases} \displaystyle \frac{\varphi(n)}{\varphi(e)}\chi^*(s),
& \text{ if $d\mid e$ and $(s,e)=1$}, \\ 0, & \text{ otherwise},
\end{cases}
\end{equation*}
where $\chi^*$ is the primitive character \textup{(mod $d$)} that induces $\chi$.
\end{lemma}

\begin{proof}[Proof of Lemma {\rm \ref{Lemma_charact}}]
We can assume $(a,n)=1$ in the sum. If $a\equiv s$ (mod $e$), then $(s,e)=(a,e)=1$. Given the Dirichlet character $\chi$ (mod $n$),
the primitive character $\chi^*$ (mod $d$) that induces $\chi$ is defined by
\begin{equation*}
\chi(a) =  \begin{cases} \chi^*(a), & \text{ if $(a,n)=1$}, \\ 0, & \text{ if $(a,n)>1$}.
\end{cases}
\end{equation*}

We deduce
\begin{equation*}
T:= \sum_{\substack{a=1\\ a\equiv s \, \text{\rm (mod $e$)} }}^n \chi(a) = \sum_{\substack{a=1\\ (a,n)=1\\ a\equiv s \,
\text{\rm (mod $e$)} }}^n \chi^*(a)
= \sum_{r=1}^d  \chi^*(r) \sum_{\substack{a=1\\ (a,n)=1\\ a\equiv r \, \text{\rm (mod $d$)}\\ a\equiv s \, \text{\rm (mod $e$)} }}^n 1,
\end{equation*}
where the inner sum is evaluated in Lemma \ref{Lemma_cong}. Since $(s,e)=1$, as mentioned above, we have
\begin{equation*}
T=  \sum_{\substack{r=1\\ (r,d)=1\\ (d,e)\mid r-s}}^d  \chi^*(r) \frac{\varphi(n)}{\varphi(de)}(d,e)=
\frac{\varphi(n)}{\varphi(de)}(d,e) \sum_{\substack{r=1\\ (r,d)=1\\ r\equiv s \, \text{\rm (mod $(d,e)$)}}}^d  \chi^*(r)
= \frac{\varphi(n)}{\varphi(de)}(d,e) \chi^*(s),
\end{equation*}
by Lemma \ref{Lemma_char_sum} in the case $(d,e)=d$, that is $d\mid e$. We conclude that
\begin{equation*}
T=\frac{\varphi(n)}{\varphi(de)}d \chi^*(s)= \frac{\varphi(n)}{\varphi(e)} \chi^*(s).
\end{equation*}

If $d\nmid e$, then $T=0$.
\end{proof}

\begin{proof}[Proof of Theorem {\rm \ref{Th_gen_TL}}]
Let $V$ denote the given sum. By using the identity $F(n)=\sum_{e\mid n} (\mu*F)(e)$, we have
\begin{gather*}
V= \sum_{a_1,\ldots,a_m, b_1,\ldots,b_k=1}^n \chi_1(a_1)\cdots \chi_m(a_m) \lambda_1(b_1)\cdots \lambda_k(b_k)
\sum_{e\mid (a_1-s_1,\ldots, a_m-s_m, b_1,\ldots,b_k,n)} (\mu*F)(e)
\end{gather*}
\begin{gather*}
= \sum_{e\mid n} (\mu*F)(e) \sum_{\substack{a_1=1\\ a_1\equiv s_1 \, \text{\rm (mod $e$)} }}^n \chi_1(a_1)
\cdots \sum_{\substack{a_m=1\\ a_m\equiv s_m \, \text{\rm (mod $e$)} }}^n \chi_m(a_m)
\sum_{\substack{b_1=1\\ e\mid b_1}}^n \lambda_1(b_1)\cdots \sum_{\substack{b_k=1\\ e\mid b_k}}^n \lambda_k(b_k)
\end{gather*}

Here for every $1\le \ell \le k$,
\begin{gather*}
\sum_{\substack{b_{\ell}=1\\ e\mid b_{\ell}}}^n \lambda_{\ell}(b_{\ell})= \sum_{c_{\ell}=1}^{n/e} \exp(2\pi i w_{\ell}c_{\ell}/(n/e))=
\begin{cases} \frac{n}{e}, & \text{ if $\frac{n}{e}\mid w_{\ell}$}, \\ 0, & \text{ otherwise},
\end{cases}
\end{gather*}
and using Lemma \ref{Lemma_charact} we deduce that
\begin{gather*}
V= \chi_1^*(s_1)\cdots \chi_m^*(s_m) \sideset{}{'}\sum  (\mu*F)(e) \left(\frac{\varphi(n)}{\varphi(e)}\right)^m \left(\frac{n}{e}\right)^k,
\end{gather*}
where the sum $\sum^{'}$ is over $e\mid n$ such that $d_j\mid e$, $(e,s_j)=1$ for all $1\le j\le m$ and
$n/e\mid w_{\ell}$ for all $1\le \ell \le k$.
Interchanging $e$ and $n/e$, the sum is over $e$ such that $e\mid n/d_j$, $(n/e,s_j)=1$ for all $1\le j\le m$ and $e\mid w_{\ell}$
for all $1\le \ell \le k$. This completes the proof.
\end{proof}

\section{Acknowledgement} This work was supported by the European Union, co-financed by the European
Social Fund EFOP-3.6.1.-16-2016-00004.

\end{document}